\documentclass{amsart}
\usepackage[utf8]{inputenc}
\usepackage{amsmath,times,hyperref, multicol}
\usepackage{amsthm}
\usepackage{amssymb}
\usepackage{amsfonts}
\usepackage{latexsym}
\usepackage{graphicx}
\usepackage[usenames,dvipsnames]{color} 
\usepackage{enumitem}
\setlist[itemize]{leftmargin=*}
\setlist[enumerate]{leftmargin=*}
\usepackage[font={small}]{caption}
\usepackage{subcaption}
\usepackage{cite}
\usepackage{mathtools}
\mathtoolsset{showonlyrefs}

\makeatletter
\newenvironment{leftalign*}[1][\parindent]{\setlength\hangindent{#1}\start@align\tw@\st@rredtrue\m@ne}{\endalign}

\makeatother

\usepackage{tikz,amsmath,amsfonts}
\usetikzlibrary{calc,intersections}
\usetikzlibrary{decorations.pathreplacing}
\usepackage{ifthen}
\renewcommand{\phi}{\varphi}

\newcommand{\rr}{\mathbb{R}}
\newcommand{\R}{\mathfrak{R}}
\newcommand{\I}{\mathfrak{I}}

\newcommand{\C}{\mathbb{C}}

\newcommand{\N}{\mathbb{N}}

\newcommand{\norm}[1]{\| #1 \|}

\newcommand{\hs}{\text{Haus}}

\newcommand{\dn}{\mathbf{dn}}

\newcommand{\Latremoliere}{Latr\'emoli\`ere}

\let\oldenumerate=\enumerate
	\def\enumerate{
	\oldenumerate
	\setlength{\itemsep}{5pt}
	}
\let\olditemize=\itemize
	\def\itemize{
	\olditemize
	\setlength{\itemsep}{5pt}
	}

\theoremstyle{plain}
\newtheorem{Theorem}{Theorem}[section]

\newtheorem{Proposition}[Theorem]{Proposition}

\theoremstyle{definition}
\newtheorem{Convention}[Theorem]{Convention}
 \newtheorem{Definition}[Theorem]{Definition}
\newtheorem{Notation}[Theorem]{Notation}

\allowdisplaybreaks

\numberwithin{equation}{section}

\begin{document}

\title[The Fell topology and the modular Gromov-Hausdorff propinquity]{The Fell topology and the modular Gromov-Hausdorff propinquity}
\author{Konrad Aguilar}
\address{Department of Mathematics and Statistics, Pomona College, 610 N. College Ave., Claremont, CA 91711} 
\email{konrad.aguilar@pomona.edu}
\urladdr{\url{https://aguilar.sites.pomona.edu/}}
\thanks{The first author gratefully acknowledges the financial support from the Independent Research Fund Denmark through the project 'Classical and Quantum Distances" (grant no. 0940-00107B)}

\author{Jiahui Yu}
\address{Department of Mathematics and Statistics, Pomona College, 610 N. College Ave., Claremont, CA 91711} 
\email{jyad2018@mymail.pomona.edu}

\date{\today}

\begin{abstract}
  Given a unital AF-algebra $A$ equipped with a faithful tracial state, we equip each (norm-closed two-sided) ideal of $A$ with a metrized quantum vector bundle structure, when canonically viewed as a module over $A$, in the sense of \Latremoliere{} using previous work of the first author and \Latremoliere{}. Moreover, we show that convergence of ideals in the Fell topology implies convergence of the associated  metrized quantum vector bundles in the modular Gromov-Hausdorff propinquity of \Latremoliere{}.
  In a similar vein but requiring a different approach, given a compact metric space $(X,d)$, we equip each ideal of $C(X)$ with a metrized quantum vector bundle structure, and show that convergence in the Fell topology implies convergence in the modular Gromov-Hausdorff propinquity. 
\end{abstract}
\maketitle

\section{Introduction and Background}

The realm of noncommutative metric geometry introduced by Rieffel \cite{Rieffel98a,Rieffel00} provides a metric framework for continuity of C*-algebras using noncommutative analogues to compact metric spaces and the Gromov-Hausdorff distance. Examples of C*-algebras that vary continuously in certain senses include: continuous fields of C*-algebras and in particular  inductive/direct limits of C*-algebras, matrices converging to the sphere and other structures found in the high-energy physics literature, etc \cite{Rieffel01}. Rieffel was first to introduce a noncommutative analogue to the Gromov-Hausdorff distance called the Gromov-Hausdorff distance \cite{Rieffel00}, and many others developed their own noncommutative analogues to answer various questions left after Rieffel developed his \cite{Kerr02, Kerr09, Li03, Wu, Latremoliere13, Latremoliere13c}. Our work has led us to work mainly with \Latremoliere's distance, the Gromov-Hausdorff propinquity \cite{Latremoliere13, Latremoliere13c}, due to the advantages it has afforded the first author -  in particular, the completeness argument in \cite{Latremoliere13c}
as applied in \cite{Aguilar18, Aguilar21}. Moreover, with \Latremoliere, the first author also established convergence of AF-algebras using the Gromov-Hausdorff propinquity.

However, there are many   important structures one can associate to C*-algebras  such as group actions, Hilbert C*-modules, spectral triples, etc, whose continuity is also of interest, and  \Latremoliere{} has developed a suite of quantum distances to address each of these scenarios \cite{Latremoliere18, Modular, Latremoliere21b, Latremoliere21} and has answered fascinating continuity results of these structures \cite{Latremoliere2020b}. Rieffel has also utilized \Latremoliere's Gromov-Hausdorff propinquity to show convergence of various structures as well \cite{Rieffel15, Rieffel18}. 

In this article, we focus on convergence of certain Hilbert C*-modules. Although  quantum distances require more information than that of the C*-algebras  such as L-seminorms (noncommutative analogues to the Lipschitz seminorm), it is natural to consider structures arising purely from the C*-algebraic structure. For instance, when \Latremoliere{} introduced the modular Gromov-Hausdorff propinquity, he first proved that convergence of a sequence of quantum compact metric spaces in the Gromov-Hausdorff propinquity implies convergence of free modules of finite rank in the modular Gromov-Hausdorff propinquity, and second, he provided sufficient conditions for when convergence of two sequences of  modules implies convergence of their direct sum \cite[Section 7 and 8]{Modular}. Following in this theme, we look at ideals of C*-algebras in this article, which canonically form modules over C*-algebras, and ask if convergence of ideals in the Fell topology provide convergence in the modular Gromov-Hausdorff propinquity. In Section \ref{s:AF}, we answer this in the positive for ideals of any unital  AF-algebra equipped with a faithful tracial state. In Section \ref{s:comm}, we answer this in the positive for any separable unital commutative C*-algebra (note that a unital   C*-algebra equipped with a quantum metric is necessarily separable, so this case is as general as possible).  For the remainder of this section, we list the necessary definitions and results for the rest of the article. 
We note that all convergence result in this paper are also true for the dual modular Gromov-Hausdorff propinquity since the modular propinquity dominates it (see  \cite[Proposition 3.17]{Latremoliere21b}). 

\begin{Convention}\label{c:c*}
Let $A$ be a unital C*-algebra. We denote its norm by $\|\cdot\|_A$ and identity by $1_A$ unless otherwise specified. Moreover, by an {\em ideal } of a C*-algebra, we mean a norm-closed two-sided ideal, and we denote the set of ideals by $Ideals(A)$.
\end{Convention}

The story begins with quantum compact metric spaces.

\begin{Definition}{\cite[Definition 2.3]{Modular}}\label{d:admissible}
A function $F:[0, \infty)^4 \rightarrow [0, \infty)$ is {\em admissible} when for all
\[
(x_1, x_2, x_3, x_4), (y_1, y_2, y_3, y_4)\in [0,\infty)^4
\]
such that $x_j\leq y_j$ for all $j \in \{1,2,3,4\}$, we have
\[
F(x_1, x_2, x_3, x_4)\leq F (y_1, y_2, y_3, y_4)
\]
and $x_1x_4+x_2x_3 \leq F(x_1, x_2, x_3, x_4)$.
\end{Definition}

\begin{Definition}{\cite[Definition 2.4]{Modular}}\label{d:qcms}
An {\em $F$-Leibniz quantum compact metric
space} $(A, L)$, for some admissible function $F$, is a unital C*-algebra $A$ and a seminorm $L$ defined on a dense Jordan-Lie subalgebra $dom (L)$ of $sa (A)$, such that
\begin{enumerate}
\item  $\{a \in dom (L) : L(a) = 0\} = \rr1_A$,
\item  the Monge-Kantorovich metric $mk_L$, defined on the state space $S (A)$ by
setting for all $\phi, \psi \in S (A)$
\[mk_L(\phi, \psi) = \sup \{|\phi(a)-\psi(a)| : a \in sa (A), L(a) \leq 1\} ,\]
metrizes the weak* topology on $S (A)$,
\item  $L$ is lower semi-continuous, i.e. $\{a \in sa (A) : L(a) \leq  1\}$ is norm closed,
\item  for all $a, b \in dom (L)$, we have:
\[\max \{L (a \circ b) , L (\{a, b\})\} \leq  F (\|a\|_A, \|b\|_A, L(a), L(b)) .\]
\end{enumerate}
The seminorm $L$ of a $F$-Leibniz quantum compact metric space $(A, L)$ is
called an {\em L-seminorm} of type $F$.
\end{Definition}

The two main examples of quantum metric spaces we focus on in this article are given in the following two theorems.

\begin{Theorem}{\cite{Kantorovich40, Kantorovich58}}\label{t:cx-lip}
Let $(X,d)$ be a compact metric space. If we define
\[
L_d(f)=\sup \left\{\frac{|f(x)-f(y)|}{d(x,y)} : x,y \in X, x \neq y \right\}
\]
for all $f \in C(X)$, then $(C(X), L_d)$ is an $F$-Leibniz compact quantum metric space, where $F(x_1, x_2, x_3, x_4)=x_1x_4+x_2x_3$ for all $(x_1,x_2, x_3, x_4)\in [0,\infty)^4$.
\end{Theorem}

\begin{Theorem}{{\cite[Theorem 3.5]{Aguilar}}}\label{t:af-lip}
Let $A=\overline{\cup_{n \in \N}A_n}^{\|\cdot\|_A}$ be a unital AF-algebra equipped with a faithful tracial state $\tau$ such that $A_0=\C1_A$. For each $n \in \N$, let 
\[
E_n:A \rightarrow A_n
\]
denote the unique $\tau$-preserving conditional expectation onto $A_n$. Let $(\beta(n))_{n \in \N}$ be a sequence in $(0,\infty)$ that converges to $0$.   For each $(x_1,x_2, x_3, x_4)\in [0,\infty)^4$, set \[F(x_1, x_2, x_3, x_4)=2(x_1x_4+x_2x_3).\]

If we define
\[
L_\beta (a)=\sup \left\{ \frac{\|a-E_n(a)\|_A}{\beta(n)} : n \in \N \right\}
\]
for all $a \in A$, then $(A, L)$ and    $(A_n, L)$ for all $n \in \N$ are $F$-Leibniz   quantum compact metric spaces such that 
\[
\Lambda ((A_n, L), (A, L)) \leq \beta(n)
\]
for all $n \in \N$, and thus
\[
\lim_{n \to \infty} \Lambda ((A_n, L), (A, L))=0,
\]
where $\Lambda$ is the quantum Gromov-Hausdorff propinquity of \cite{Latremoliere13}.
\end{Theorem}

Now, we move on to the module setting.

\begin{Definition}{{\cite[Definition 3.6]{Modular}}}\label{d:fgh-admissible}
A triple $(F,G,H)$ is {\em admissible} when
\begin{enumerate}
    \item $F:[0, \infty)^4\rightarrow [0,\infty)$ is admissible in the sense of Definition \ref{d:admissible},
    \item $G:[0,\infty)^3\rightarrow [0,\infty)$ satisfies $G(x,y,z)\leq G(x',y',z')$ if $x\leq x', y\leq x', z\leq z'$, while
    \[
    (x+y)z\leq G(x,y,z), 
    \]
    and
    \item $H:[0,\infty)^2\rightarrow [0,\infty)$ satisfies $H(x,y)\leq H(x',y')$ if $x\leq x', y \leq y'$ while $2xy\leq H(x,y).$
\end{enumerate}
\end{Definition}

\begin{Notation}
Let $a \in A$, where $A$ is a C*-algebra. We denote $\R a=\frac{a+a^*}{2}\in sa(A)$ and $\I a=\frac{a-a^*}{2i}\in sa(A)$, where $a=\R a+i\I a$.
\end{Notation}

\begin{Definition}{{\cite[Definition 3.8]{Modular}}}\label{d:mqvb}
An   {\em $(F,G,H)$-metrized quantum vector bundle} \[\Omega=(M, \langle \cdot, \cdot \rangle_M, D_M, A, L)\] for some admissible triple $(F,G,H)$  is given by an $F$-Leibniz quantum compact metric space $(A, L)$ as well as a left Hilbert module $(M, \langle \cdot, \cdot\rangle_M)$ over $A$ (see \cite[Definition 3.4]{Modular} and \cite{Rieffel74}) and a norm $D_M$ defined on a dense $\C$-subspace $dom(D_M)$ of $M$ such that 
\begin{enumerate}
    \item $\|\cdot\|_M\leq D_M$, where $\|\cdot\|_M=\sqrt{\|\langle \cdot, \cdot\rangle_M\|_A}$ is the norm induced by $\langle \cdot, \cdot\rangle_M$,
    \item the set
    \[
    \{ \omega \in M: D_M(\omega)\leq 1\}
    \]
    is compact for $\|\cdot\|_M$
    \item for all $a \in sa(A)$ and for all $\omega \in M$, we have:
    \[
    D_M(a\omega) \leq G(\|a\|_A, L(a), D_M(\omega)),
    \]
    \item for all $\omega, \eta \in M$, we have
    \[
    \max\{ L(\R\langle \omega, \eta\rangle_M), L(\I \langle \omega, \eta\rangle_M)\}\leq H(D_M(\omega), D_M(\eta)).
    \]
\end{enumerate}
    The norm $D_M$ is called the {\em D-norm} of $\Omega$.
\end{Definition}

The modular Gromov-Hausdorff propinquity, $\Lambda^{mod}$, introduced by \Latremoliere{} in \cite{Modular} forms a metric on the class of $(F,G,H)$-metrized quantum vector bundles up to a natural notion of isomorphism (see \cite[Definition 3.18, 3.19]{Modular}). The rest of the article concerns itself by establishing natural convergence results with respect to $\Lambda^{mod}$. We do not define $\Lambda^{mod}$ in this article since we do not require its entire construction to achieve our results.

\section{The AF-algebra case}\label{s:AF}

We now establish that for a given unital AF algebra equipped with faithful tracial state, it holds that convergence of ideals in the Fell topology implies convergence of modules equipped with appropriate D-norms. The strategy is to use the $L$-seminorm of Theorem \ref{t:af-lip} and introduce a quantity that captures the ideal structure. This is done by canonical approximate identities for ideals of AF algebras. This approach allows us establishing finite-dimensional approximations to achieve our results. We also note that  have to restrict the sequence $(\beta(n))_{n \in \N}$ of Theorem \ref{t:af-lip} in order to allow for finite-dimensional approximations and note that   $(1/(n!))_{n \in \N}$ satisfies the following restriction on $(\beta(n))_{n \in \N}.$

\begin{Proposition}\label{idealnAF}
Using the setting of Theorem \ref{t:af-lip}, assume furthermore that 
\begin{equation}\label{eq:beta} \left(\frac{\beta(n)}{\beta(n-1)}\right)_{n \in \N}\end{equation}
 converges to $0$. 
Let $I \in \text{Ideals}(A)$. For any $n \in \N$, put $I_n=I\cap A_n$. Since $I_n$ is finite dimensional,   let $1_n$ be the unit in $I_n$ (note that $(1_n)_{n \in \N}$ is the canonical  approximate identity for $I$). For any $\omega \in I$, define
$$D_I(\omega)=\max\left\{L_\beta(\omega),\norm{\omega}_A,\sup_{n\in \N}\left\{\frac{\norm{\omega-\omega 1_n}_A}{\beta(n)} \right\}\right\}.$$
Let $D_{I_n}$ be the restriction of $D_I$ to $I_n$. Let $G(x,y,z)=F(x,y,z,z)$ and $H(x,y)=F(x,x,y,y)$ for all $(x,y,z)\in [0, \infty)^3$ with $F$ as in Theorem \ref{t:af-lip}. Then, $(F,G,H)$ is an admissible triple. Define $$\langle \omega, \nu \rangle_{I}=\omega \nu^* $$
for all $\omega, \nu \in I$, 
and let  $\langle \cdot, \cdot \rangle_{I_n}$ denote  the restriction of 
$\langle \cdot , \cdot \rangle_{I}$ to $I_n$.

It holds that  $\Omega^\beta_I=(I, \langle \cdot, \cdot \rangle_I, D_I, A, L_\beta)$ and $\Omega^\beta_{I_n}=(I_n, \langle \cdot, \cdot \rangle_{I_n}, D_{I_n}, A_n, L_\beta)$ are $(F,G,H)$-metrized quantum vector bundles for any $n\in \N$ such that 
$$\lim_{n\rightarrow\infty}\Lambda^{mod}(\Omega^\beta_{I_n}, \Omega^\beta_I)=0,$$
where $\Lambda^{mod}$ is the modular Gromov-Hausdorff propinquity of  \cite{Modular}.

In particular, for each $n \in \N, n \geq 1$, it holds that
\[
 \Lambda^{mod}(\Omega^\beta_{I_n}, \Omega^\beta_I)\leq \max\{\beta(n), x_n-1\},
\]
where $x_n=\max\{1+\beta(n), x_n'\}$ with $x_n'=\max_{m<n}\left\{\frac{2\beta(n)+\beta(m)}{\beta(m)} \right\}$, and $( \max\{\beta(n), x_n-1\})_{n \in \N}$ converges to $0$ as $n \to \infty$.
\end{Proposition}

\begin{proof}
 For ease of notation in this proof, we use $L$ for $L_\beta$,  $D$ for $D_I$, $D_n$ for $D_{I_n}$,   $\norm{\cdot}$ for $\norm{\cdot}_A$, $\Omega$ for $\Omega^\beta_I$, and $\Omega_n$ for $\Omega^\beta_{I_n}$. Moreover,  let $B=\{\omega \in I: D(\omega) \leq 1\}$ and $B_n=B\cap I_n$. 
We let 
$$\gamma_n=\{\Omega_n, \Omega, A, 1_A, \iota_n, id_A, (\omega_j)_{j\in J}, (\omega_j)_{j\in J} \}$$
be the bridge (see \cite[Definition 4.4]{Modular}) between $\Omega_n$ and $\Omega$, where 
      $\iota_{n}: A_n \rightarrow A$ is the inclusion map, $id_A:A\rightarrow A$ is the identity map on $A$, and $J=B_n$ and  for any $j \in J$, we let $\omega_j=\nu_j=j$. Thus, $ \{\omega_j:j\in J\}=\{\nu_j:j\in J\}=B_n.$

We shall break the proof of the proposition into several steps. First, we verify that $\Omega$ and $\Omega_n$ are $(F,G,H)$-metrized quantum vector bundles. Then, we verify that $\lambda(\gamma_n)\rightarrow0$ as $n\rightarrow\infty$ by bounding the basic reach $\varrho_b(\gamma_n)$, the height $\varsigma(\gamma_n)$,  the modular reach $\varrho^{\#}(\gamma_n)$ and the imprint $\varpi(\gamma_n)$ (see \cite[Definition 4.18, 4.10, 4.12, 4.14, and 4.15]{Modular} for the definitions of $\lambda$,  $\varrho_b$,  $\varsigma$,  $\varrho^{\#}$,   $\varpi$, respectively).

  \noindent \textbf{Step 1: We first show that $\Omega$ and $\Omega_n$ are $(F,G,H)$-metrized quantum vector bundles.} We show this for $\Omega$ and the result for $\Omega_n$ follows immediately. 
  
 First,  $\norm{\omega}\leq D(\omega)$ for all  $\omega \in I$  by definition of $D$. Next,  By \cite[Theorem 1.9]{Rieffel} and Theorem \ref{t:af-lip}, the set  $\{\omega\in A: \norm{\omega} \leq 1 , L(\omega)\leq 1\}$ 
        is compact. Note that 
        $$\left\{\omega\in A: \sup_{n\in \N}\left\{\frac{\norm{\omega-\omega 1_n}}{\beta(n)} \right\}\leq 1\right\}$$
        is closed, $I$ is closed and the intersections of compact sets with closed sets are compact. Thus, $B$ is compact. 
        Next, let $a \in sa(A)$ and  $\omega \in I$. We have
\begin{align*}
 L(a\omega) &\leq  2(\norm{a}L(\omega)+\norm{\omega}L(a))  \leq  2(\norm{a} D(\omega) + L(a)D(\omega))\\
 &= F(\norm{a}, L(a), D(\omega),D(\omega) )=G(\norm{a}, L(a), D(\omega)).
 \end{align*}
    Moreover, 
  \[
    \norm{a\omega} \leq  \norm{a}\norm{\omega} \leq \norm{a} D(\omega) \leq   2(\norm{a}+L(a))D(\omega)  =  G(\norm{a}, L(a), D(\omega))\]
    Similarly, for any $n \in \N$
    \[
    \frac{\norm{a\omega-a\omega1_n}}{\beta(n)} \leq  \norm{a}\cdot \frac{\norm{\omega-\omega 1_n}}{\beta(n)} \leq   \norm{a}D(\omega) \leq   G(\norm{a}, L(a), D(\omega)).
   \]
Therefore, we have that $D(a\omega) \leq G(\norm{a}, L(a), D(\omega))$. 
Next, let $\omega, \nu \in I$.  Since $L$ is *-preserving as conditional expectations are positive, we have
\begin{align*}
 L(\langle \omega, \nu\rangle)
&=  L(\omega\nu*) \leq   2(\norm{\omega}L(\nu^*)+\norm{\nu^*}L(\omega)) =  2(\norm{\omega}L(\nu)+\norm{\nu}L(\omega))\\
&\leq   2 (D(\omega)D(\nu)+D(\omega)D(\nu))=H(D(\omega), D(\nu)).
\end{align*}
Again, as $L$ is *-preserving, we have 
\[\max\{L(\R\langle \omega, \nu \rangle), L(\I\langle \omega, \nu \rangle)\} \leq L(\langle \omega, \nu\rangle) \leq H(D(\omega), D(\nu)).\] 
   
\noindent \textbf{Step 2: bounding the height and the basic reach.}
 The basic bridge (\cite[Definition 4.8]{Modular}) for $\gamma_n$ is  
$${\gamma_n}_b=(A,1_A,\iota_n,I_A).$$
 However, this is exactly the bridge in the proof of \cite[Theorem 3.5]{Aguilar}. Thus, from that proof, we have  $\varsigma(\gamma)=\varsigma({\gamma_n}_b)=0$ and  $\varrho({\gamma_n}_b)\leq\beta(n)$.

\item \textbf{Step 3: bounding the imprint.}

The \textit{modular Monge-Kantorvich metric} (\cite[Definition 3.23]{Modular}) for our $\Omega$ is
$$k_\Omega(\omega,\nu)=\sup\{\norm{\omega^*\xi-\nu^*\xi}: \xi\in B\}$$
for any $\omega,\nu \in I$ (note that the author of \cite{Modular} uses $k_D$ and $k_\Omega$ interchangeably).  
The \textit{modular Monge-Kantorvich metric} for our $\Omega_n$ is
$$k_{\Omega_n}(\omega,\nu)=\sup\{\norm{\omega^*\xi-\nu^*\xi}: \xi\in B_n\}$$
for any $\omega,\nu \in I_n$. 
Note 
$$\norm{\omega^*\xi-\nu^*\xi}\leq \norm{\omega-\nu}\norm{\xi}\leq \norm{\omega-\nu}$$
for any $\xi \in B$. Thus, $k_{\Omega_n}(\omega,\nu)\leq \norm{\omega-\nu}$ and $k_{\Omega}(\omega,\nu)\leq \norm{\omega-\nu}$. 
The imprint $\varpi(\gamma_n)$ is 
$$\varpi(\gamma_n) = \max\{\hs_{k_{\Omega_n}}(\{\omega_j:j\in J\}, B_n), \hs_{k_{\Omega}}(\{\nu_j:j\in J\}, B)\}.$$
Note that 
$$\hs_{k_{\Omega_n}}(\{\omega_j:j\in J\}, B_n)= \hs_{k_{\Omega_n}}(B_n, B_n)=0$$
and 
$$\hs_{k_{\Omega}}(\{\nu_j:j\in J\}, B_n)= \hs_{k_{\Omega}}(B_n, B)\leq \hs_{\norm{\cdot}}(B_n, B).$$
Our goal now is to show that 
$$\lim_{n\rightarrow\infty}\hs_{\norm{\cdot}}(B_n, B)=0.$$
Note that $\sup_{b\in B_n}d(b,B) =0$ as $B_n \subseteq B$. 
Thus, 
$$\hs_{\norm{\cdot}}(B_n, B)=\max\{\sup_{b\in B_n}d(b,B),\sup_{b\in B}d(b,B_n)\}=\sup_{b\in B}d(b,B_n).$$
Let $b\in B$, then 
$$\norm{b}\leq 1, L(  b)\leq 1 \quad \text{and} \quad  \sup_{m\in \N}\left\{\frac{\norm{b-b1_m}}{\beta(m)}\right\}\leq 1.$$
 Thus, for any $m\in \N$, $\norm{b-E_m(b)}\leq \beta(m)$ and  $\norm{b-b1_m}\leq \beta(m)$. 
Note that $E_m(b1_m)=E_m(b)1_m\in I_m$ as $E_m$ is a conditional expectation on $A_m$ and  $1_m \in I_m$. Next, we shall bound $D(E_n(b)1_n)$ for all $n \in \N$. Let $n \in \N.$ Then
      $$\norm{E_n(b1_n)}\leq \norm{b1_n}\leq \norm{b}\norm{1_n} \leq 1$$
    since $1_n$ is a projection. Fix $m \in \N$, 
    if $m\geq n$,$$\norm{E_n(b)1_n-E_m(E_n(b)1_n)}=0.$$
    If $m< n$, 
    \begin{align*}
     \norm{E_n(b1_n)-E_m(E_n(b1_n))} 
    &= \norm{E_n(b1_n)-E_n(E_m(b1_n))} \\&=  \norm{E_n(b1_n-E_m(b1_n))} \leq  \norm{b1_n-E_m(b1_n)}\\
    &\leq  \norm{b-b1_n}+\norm{b-E_m(b)}+\norm{E_m(b)-E_m(b1_n)}\\
    &\leq  \norm{b-b1_n}+\norm{b-E_m(b)}+\norm{b-b1_n}\\
    &\leq  2\beta(n)+\beta(m)
    \end{align*}
    Fix $m\in \N$, if $m\geq n$,
    $$\norm{E_n(b1_n)-E_n(b)1_n1_m}=0.$$
If $m<n$, 
    \begin{align*}
     \norm{E_n(b1_n)-E_n(b)1_n1_m} 
    &=  \norm{E_n(b1_n)-E_n(b)1_m}\\
    &\leq  \norm{E_n(b)-E_n(b1_n)}+ \norm{E_n(b)-E_n(b1_m)}\\
    &\leq  \norm{b-b1_n}+ \norm{b-b1_m}\\
    &\leq  \beta(n)+\beta(m)< 2\beta(n)+\beta(m)
    \end{align*} 
Put $x'_n=\max_{m<n}\left\{\frac{2\beta(n)+\beta(m)}{\beta(m)}\right\}$ and 
 $x_n=\max\{1+\beta(n), x'_n\}.$  
Note that  $x_n \rightarrow 1$ as $n\rightarrow \infty$  by Expression \eqref{eq:beta}. Moreover, $x_n$ is independent of our choice of $b$. 
Then, we have that $D(E_n(b1_n))\leq x_n$. Thus,  
$$\frac{E_n(b)1_n}{x_n}\in B_n.$$
Finally, 
\begin{align*}
 \left\|b-\frac{E_n(b)1_n}{x_n}\right\| 
& \leq   \left\|E_n(b)1_n-\frac{E_n(b)1_n}{x_n}\right\|+\norm{b-E_n(b)1_n}\\
&\leq   (x_n-1)+\norm{E_n(b1_n)-E_n(b)}+\norm{E_n(b)-b}\\
&\leq   (x_n-1)+\norm{b1_n-b}+\norm{E_n(b)-b} \leq  (x_n-1)+2\beta(n).
\end{align*}
Thus, we have that $$d(b, B_n) \leq (x_n-1)+2\beta(n).$$ And so   $\hs_{k_\Omega}(B_n,B)\leq (x_n-1)+2\beta(n) $ 
which tends to zero as $n\rightarrow \infty$. 
As a result, 
$$\varpi(\gamma_n)\rightarrow 0 \text{ as } n\rightarrow \infty.$$
\item \textbf{Step 4: bounding the modular reach.}
Note that for any $j\in J$, we have that $\omega_j=\nu_j$ so the modular reach
$$\varrho^{\#}(\gamma_n)=\max\{\textbf{dn}_\gamma(\omega_j,\nu_j)\}=0$$ 
(for $\textbf{dn}_\gamma$ see \cite[Definition 4.13]{Modular}). Thus, from \cite[Definition 4.18 and 4.16]{Modular},
\begin{align*}
\lambda(\gamma_n)&= \max\{\varsigma(\gamma_n),\varrho(\gamma_n)\} = \max\{\varsigma(\gamma_n),\varrho_b(\gamma_n),\varrho^{\#}+\varpi(\gamma_n)\}\\
&= \max\{\varrho_b(\gamma_n),\varpi(\gamma_n)\} \leq \max\{\beta(n),x_n-1\}
\end{align*}
Thus, $\lambda(\gamma_n)\rightarrow0$ as $n\rightarrow\infty$. Note that
$$\Lambda^{mod}(\Omega_n,\Omega)\leq \lambda(\gamma_n)$$
by \cite[Definition 5.6]{Modular} since any bridge is a trek. 
Thus, $\Lambda^{mod}(\Omega_n,\Omega)\rightarrow 0$ as $n\rightarrow\infty$.
\end{proof}

Given these finite-dimensional approximations, we now prove that for unital AF-algebras equipped with a faithful tracial state, convergence in Fell topology implies convergence in the modular propinquity equipped with the metrized quantum vector bundle structure of the previous proposition.

\begin{Theorem}\label{t:main-af}
Using the setting of Proposition \ref{idealnAF}, we have that the map
\[
I \in Ideals(A) \longmapsto \Omega^\beta_I  
\]
is continuous with respect to the Fell topology and the topology induced by the modular Gromov-Hausdorff propinquity.
\end{Theorem}
\begin{proof}
Since the Fell topology is compact metrizable for separable unital C*-algebras  (see \cite[Lemma 3.19]{Aguilar16c}), we may prove continuity by sequential continuity.  Denote $\overline{\N}=\N\cup \{\infty\}$.
Let $(I^{n})_{n \in \overline{\N}}$ be a sequence in $Ideals(A)$ that converges in the Fell topology to  $I^\infty$. Let $\epsilon>0$. By Proposition \ref{idealnAF}, there exists $N \in \N$ such that $ \max\{\beta(N), x_N-1\}<\epsilon,$ and thus
\[
\Lambda^{mod}\left(\Omega^\beta_{I^n_N}, \Omega^\beta_{I^n} \right)\leq \max\{\beta(N), x_N-1\}<\epsilon
\]
for all $n \in \overline{\N}.$ Next, by \cite[Lemma 3.24]{Aguilar16c}, there exists $N' \in \N$ such that  for all $n \geq N'$
\[
I^n_N=I^\infty_N.
\]
Let $n \geq N'$. 
Hence, by construction of $D_{I^\infty_N}$ and $D_{I^n_N}$, we have that $
\Omega^\beta_{I^n_N}=\Omega^\beta_{I^\infty_N}.
$
Thus, by the triangle inequality for $\Lambda^{mod}$, we have
\begin{align*}
     \Lambda^{mod}\left(\Omega^\beta_{I^n}, \Omega^\beta_{I^\infty}\right)  & \leq  \Lambda^{mod}\left(\Omega^\beta_{I^n}, \Omega^\beta_{I^n_N}\right)+\Lambda^{mod}\left(\Omega^\beta_{I^n_N}, \Omega^\beta_{I^\infty_N}\right)+\Lambda^{mod}\left(\Omega^\beta_{I^\infty_N}, \Omega^\beta_{I^\infty}\right) \\
  &  <\epsilon+0+\epsilon=2\epsilon.\qedhere 
\end{align*}
\end{proof}

\section{The commutative case}\label{s:comm}
Considering the commutative case is natural due to the desire to see how the modular propinquity behaves with respect to C*-algebraic structure. However, there's another result that motivated this pursuit. Indeed, given a compact metric space $(X,d)$, denote the set of closed subsets of $X$ by $cl(X)$. It holds that the map
\begin{equation}\label{eq:comm-case}
F \in cl(X) \longmapsto (C(F), L_d)
\end{equation}
is continuous with respect to the topology induced by the Hausdorff distance and the Gromov-Hausdorff distance of \Latremoliere{} (see \cite[Theorem 6.6]{Latremoliere13} for a more general result). On the other hand, the map
\[
F \in cl(X) \longmapsto I_F:=\{f\in C(X): \forall x \in F, f(x)=0\} \in Ideals(C(X))
\]
is homeomorphism onto $Ideals(C(X))$ with respect to the topology induced by the Hausdorff distance and the Fell topology (this is a classical result, but we provide a proof of this in the following proof). Thus, the following result somewhat complements the continuity result of Expression \eqref{eq:comm-case}, and this is also displayed in the following proof since we are essentially approximating families continuous of functions that vanish on different closed sets rather than the families of continuous functions defined on those different closed sets.

\begin{Theorem}\label{t:comm-mod}
Let $(X,d)$ be a compact metric space.  Let $I \in Ideals(C(X))$. For each $ f\in I$, we define
\[
D^d_I(f)=\max\{\|f\|_{C(X)}, L_d(\R f), L_d( \I f)\}
\]
using the setting of Theorem \ref{t:cx-lip}.  
For each $f,g \in I$, we define 
\[
\langle f,g\rangle= fg^*.
\]
Let $G(x,y,z)= F(x,z,y,z)$ and $H(x,y)= F(x,y,x,y)$ for all $(x,y,z)\in [0, \infty)^3$ with $F$ as in Theorem \ref{t:cx-lip}. Then, $(F,G,H)$ is an admissible triple.

It holds that $\Omega^d_I=(I, \langle \cdot, \cdot \rangle, D^d_I, C(X), L_d)$ is an $(F,G,H)$-metrized quantum vector bundle.
Moreover, the map
\[
I \in Ideals (C(X))\longmapsto \Omega^d_I
\]
is continuous with respect to the Fell topology and the topology induced by the modular Gromov-Hausdorff propinquity. 
Furthermore, using notation introduced before the theorem, we have that
\[
F \in  cl(X)  \longmapsto \Omega^d_{I_F}
\]
is continuous with respect to the topology induced by the Hausdorff distance and the topology induced by the modular Gromov-Hausdorff propinquity.
\end{Theorem}
\begin{proof}
 In this proof, we denote $\|\cdot\|_{C(X)}$ by $\|\cdot\|$, $L_d$ by $L$, and $D^d_I$ by $D$, which causes no confusion since each of these quantities are defined on all of $C(X)$.

 \noindent \textbf{Step 1: for all $I \in Ideals(C(X))$, we   show that $\Omega^d_I$ is a $(F,G,H)$-metrized quantum vector bundle.} 
Let $I \in Ideals(C(X))$.    First,  $\norm{\omega} \leq D (\omega)$ for all  $\omega \in I$  by definition of $D$. Next,
        by \cite[Theorem 1.9]{Rieffel} or the Arzel\`a-Ascoli Theorem, the set  
         $\{\omega\in C(X): \norm{\omega} \leq 1 , L (\omega)\leq 1\}$ 
        is compact. Note that   $I$ is closed and the intersection  of compact sets with closed sets are compact. Thus, $\{f \in I: D (f)\leq 1\}$ is compact. First, we note that $(F,G,H)$ is admissible by definition. Next, let $a \in sa(A)$ and  $\omega \in I$. We have since $L$ is *-preserving,
\begin{align*}
  \max \{ L(\R a\omega), L(\I a\omega)\} 
 &\leq  L(a\omega)  \leq     F(\|a\|, \|\omega\|, L(a), L(\omega))\\
 & \leq     F(\|a\|, D(\omega), L(a),  D(\omega))  =   G(\norm{a}, L(a), D(\omega))
 \end{align*} 
 since $L$ satisfies the Leibniz rule. Moreover, 
    \[
    \norm{a\omega}  \leq  \norm{a}\norm{\omega} \leq  \norm{a} D(\omega) \leq   (\norm{a}+L(a))D(\omega) =  G(\norm{a}, L(a), D(\omega))\]
   
Therefore, we have that $D(a\omega) \leq G(\norm{a}, L(a), D(\omega))$. Next, let $\omega, \nu \in I$   Since $L$ is *-preserving, we have
\begin{align*}
 L(\langle \omega, \nu\rangle)
&=  L(\omega\nu*) \leq    \norm{\omega}L(\nu^*)+\norm{\nu^*}L(\omega)  =    \norm{\omega}L(\nu)+\norm{\nu}L(\omega) \\
&\leq    D(\omega)D(\nu)+D(\omega)D(\nu) =H(D(\omega), D(\nu)).
\end{align*}
Again, as $L$ is *-preserving, we have 
$\max\{L(\R\langle \omega, \nu \rangle), L(\I\langle \omega, \nu \rangle)\} \leq L(\langle \omega, \nu\rangle) \leq H(D(\omega), D(\nu)).$

\noindent \textbf{Step 2: We now prove continuity.} 
Since the Fell topology is compact metrizable for separable unital C*-algebras  (see \cite[Lemma 3.19]{Aguilar16c}), we may prove continuity by sequential continuity.  Let $(I_n)_{n \in \N}$ be a sequence in $Ideals(C(X))$ that converges with respect to the Fell topology to some $I \in Ideals(C(X)).$

Let $n \in \N$. 
Put  $B_{ n}=\{f \in I_n: D(f) \leq 1\}$ and $B =\{f \in I: D(f) \leq 1\}$. 
Let $$\gamma_n=\{\Omega_n, \Omega,C(X), 1_{C(X)}, id_{C(X)}, id_{C(X)}, \{\omega_j\}_{j\in J}, \{\nu_j\}_{\nu\in J} \}$$ be a bridge from $\Omega_n$ to $\Omega$ where 
 $id_{C(X)}:C(X)\rightarrow C(X)$ is the identity map from $C(X)$ to itself,    $J=B_n$ and  $\omega_j=j=\nu_j$ for $j\in B_n$.

By construction, the height and the basic reach are zero, as $I_n$ and $I$ are modules of the same $C^*$-algebra. Next, the modular reach is 
$$\max\{\dn_\gamma(\omega_j, \nu_j):j\in J \}=0.$$  as $\omega_j=\nu_j$ for $j \in J$. Thus, we only need to bound the imprint
$$\varpi(\gamma_n)=\max\{\hs_{k_D}(\{\omega_j:j\in J\}, B_{ n}), \hs_{k_D}(\{\nu_j:j\in J\}, B )\}$$
(note that the author of \cite{Modular} uses $k_D$ and $k_\Omega$ interchangeably). 
Note that 
$$\{\omega_j:j\in J\}=\{\nu_j:j\in J\}=B_n.$$
Thus, $\varpi(\gamma_n)= \hs_{k_D}(B_{ n}, B )\leq \hs_{\|\cdot\|}(B_{ n}, B )$.

Let $F_n$ be the closed subset of $X$ that corresponds to $I_n$ i.e 
$$F_n=\{x \in X: f(x)=0 \text{ for all } f \in I_n\}$$
where $I_n=\{f \in C(X): f(x)=0 \text{ for all } x \in F_n\}$, 
and let $F$ be the closed subset of $X$ that corresponds to $I$, and note that  $(I_n)_{n \in \N}$ converges to $I$ in Fell topology if and only if $(F_n)_{n \in \N}$  converges to $F$ in the Hausdorff distance with respect to $d$. Indeed, given any unital C*-algebra $A$, the Fell topology on $ Ideals (A)$ of \cite{Fell61} is given by the Fell topology of \cite{Fell62} on the closed subsets of the primitive ideal space on $A$ equipped with  the Jacobson topology along with  the canonical one-to-one correspondence between $Ideals(A)$ and closed subsets of primitive ideals in the Jacobson topology. Moreover, in our current setting, the primitive ideals of $C(X)$ are of the form $I_{\{x\}}=\{f \in C(X): f(x)=0\}$ for all $x \in X$,   and the set of primitive ideals $\{I_{\{x\}} \subseteq C(X) : x \in X\}$ equipped with the Jacobson topology  is homeomorphic to $(X,d)$ (this is a well-known result, which follows from \cite[Proposition 4.3.3]{Pedersen79}). Finally, the Fell topology of \cite{Fell62} on the closed subsets of $(X,d)$ is homeomorphic to the topology on the closed subsets of $(X,d)$ induced by the Hausdorff distance by \cite[Corollary 5.1.11]{Beer}.

Thus, let $\epsilon>0$ and without loss of generality  assume that $\epsilon<1$, then  there exists $N \in \N$ such that for any $n\geq N$, we have 
$$\hs(F_n, F) < \epsilon^2.$$
Let $n \geq N.$ Let $f \in B_{ n}$. Then, $f(x)=0$ for  $x \in F_n$, $L(\R f)\leq 1$, $L(\I f)\leq 1$, and $\norm{f}\leq 1$.  Define
$$G_n=\{x\in X: d(x, F_n \cup F)<\epsilon\}.$$
Then $G_n^c$ is closed.
Note that $(F\cup F_n)\cap G_n^c=\emptyset$, and thus we may define a real-valued function $f_1$ on $F \cup F_n \cup G_n^c$ by 
$$f_1 =  \frac{1}{1+\epsilon}\R f\cdot \chi_{G_n^c}  $$
where  $\sup_{x \in F\cup F_n \cup G_n^c}|f(x)|\leq \norm{f}  \leq 1$.

\noindent\textbf{Step 2a: we show $L(f_1) \leq 1$.}
Let us discuss the following three cases. 
 First, $x,y \in F\cup F_n$.
    By the definition of $f_1$, $f_1(x)=f_1(y)=0$, so $\frac{|f_1(x)-f_1(y)|}{d(x,y)} =0.$

Second, $x,y \in G_n^c$. 
    $$\frac{|f_1(x)-f_1(y)|}{d(x,y)} =\frac{1}{1+\epsilon}\cdot \frac{|\R f(x)-\R f(y)|}{d(x,y)}\leq 1.$$

Third,    $x \in F \cup F_n$ and $y\in G_n^c$. By the definition of $G_n$, $d(x,y) \geq \epsilon$. As $\hs_d(F_n, F) < \epsilon^2$, there exists $z\in F_n$ such that $d(x,z)\leq \epsilon^2$, and so $\frac{d(x,z)}{d(x,y)}\leq \epsilon$. As $L(\R f) \leq 1$, we have 
    $$\frac{|\R f(y)|}{d(y,z)}=\frac{|\R f(y)-\R f(z)|}{d(y,z)}\leq 1.$$
    Thus, $|\R f(y)|\leq d(y,z)$. 
    Therefore,  
    \begin{align*}\frac{|f_1(x)-f_1(y)|}{d(x,y)} &= \frac{|f_1(y)|}{d(x,y)} = \frac{1}{1+\epsilon}\cdot\frac{|\R f(y)|}{d(x,y)}  \leq  \frac{1}{1+\epsilon}\cdot \frac{d(y,z)}{d(x,y)}\\
    &\leq  \frac{1}{1+\epsilon}\cdot \frac{d(y,x)+d(x,z)}{d(x,y)} \leq  \frac{1}{1+\epsilon}\left( 1+\frac{d(x,z)}{d(x,y)}\right) 
    \leq 1.
    \end{align*} 
 
Note that $f_1$ is Lipschitz function such that  $L(f_1)\leq 1$ and is defined on the closed subset $F\cup F_n \cup G_n^c$. Thus,
we can apply \cite[Theorem 1]{mcshane} and \cite[Corollary 2]{mcshane} to extend   $f_1$ to a real-valued function $g_1$ defined on $X$ that has the same Lipschitz constant and  norm as $f$ such that:  $g_1$ is  zero on $F\cup F_n$,  $g_1=\frac{1}{1+\epsilon}\R f$ on $G_n^c$, 
      $L(g_1)\leq 1$,  $\norm{g_1} \leq \norm{\R f}$. 
Similarly, we can define 
  $f_2$ on $F \cup F_n \cup G_n^c$ as 
$$f_2 =  \frac{1}{1+\epsilon}\I f \cdot \chi_{G_n^c}$$
such  that $L(f_2) \leq 1$ and obtain a real-valued extension $g_2$ of $f_2$ to $X$ such that   $g_2$ is   zero on $F\cup F_n$, 
     $g_2=\frac{1}{1+\epsilon}\I f$ on $G_n^c$, 
      $L(g_2)\leq 1$, and $\norm{g_2} \leq \norm{\I f}.$ Put $h=g_1+ig_2$.

\noindent\textbf{Step 2b: we show  
$h \in B $.} By the definition of $h$, $L(\R h)\leq 1$ and $L(\I h) \leq 1$. 
Moreover $$\norm{h}\leq \sqrt{\norm{g_1}^2+\norm{g_2}^2}\leq \sqrt{\norm{\R f}^2+\norm{\I f}^2}=\norm{f} \leq 1.$$
As $g_1(x)=g_2(x)=0$ for $x \in F\cup F_n$, we have $h(x)=0$ for $x\in F\cup F_n$. By the definition of $B  $, we have $h\in B  $.

\noindent\textbf{Step 2c: we show
$\norm{f-h}\leq 8 \epsilon.$}  
  First, assume $x\in G_n^c$.
     Thus 
    $$|\R f(x)-f_1(x)|=\left|\R f(x)-\frac{1}{1+\epsilon}\R f(x)\right|\leq \frac{\epsilon}{1+\epsilon}|f(x)|\leq \epsilon $$
    and similarly $|\I f(x) -f_2(x)|\leq \epsilon$.
    Then, 
    \begin{align*}
    |f(x)-h(x)|&=  (|\R f(x)-f_1(x)|^2+|\I f(x)-f_2(x)|^2)^\frac{1}{2} \leq   \sqrt{2}\epsilon<8 \epsilon.
    \end{align*}
 Second, assume $x \in G_n$. Since $d(x, F\cup F_n)<\epsilon$, there exists $y \in F\cup F_n$ such that $d(x,y)<\epsilon$. Since $\hs_d(F_n,F)<\epsilon^2<\epsilon$, there exists $z \in F_n$ such that $d(y,z)<\epsilon$. Hence $d(x,z)<2\epsilon$.  Hence as $\R f(z)=0$, we have
\[
|\R f(x)|=d(x,z)\cdot \frac{|\R f(x)-\R f(z)|}{d(x,z)}\leq d(x,z)\cdot 1<2 \epsilon.
\]
A symmetric argument  shows that $|g_1(x)|<2\epsilon.$ Thus, by the triangle inequality, $|\R f(x)-g_1(x)|< 4\epsilon$. Similarly, we have $|\I f(x)-g_2(x)| < 4\epsilon$.
Thus, 
$$|f(x)-h(x)|=(|\R f(x)-g_1(x)|^2+|\I f(x)-g_2(x)|^2)^{1/2}\leq 4\sqrt{2}\epsilon < 8 \epsilon.$$ 
Therefore, $|f(x)-h(x)|<8\epsilon$ for all $x \in X$, which completes this step. 

Now, we complete \textbf{Step 2}. 
From \textbf{Step 2b and 2c}, we have $d(f,B)\leq 8\epsilon$. Similarly, we can show that for any $f \in B$, $d(f, B_{n})\leq 8\epsilon$. Thus, 
$\hs_{\norm{\cdot}}(B_{n}, B)\leq 8\epsilon$. As $k_D(f) \leq \norm{f}$ for any $f \in C(X)$, we have that 
$\hs_{k_D}(B_{n}, B)\leq 8\epsilon$. Thus, $\lambda(\gamma_n)\rightarrow 0$ as $n \rightarrow \infty$. Note that
$$\Lambda^{mod}(\Omega^d_{I_n},\Omega^d_I)\leq \lambda(\gamma_n)$$
by \cite[Definition 5.6]{Modular} since any bridge is a trek. 
Thus, $\Lambda^{mod}(\Omega^d_{I_n},\Omega^d_I)\rightarrow 0$ as $n\rightarrow\infty$ which establishes the first continuity result in the statement of the theorem.
The last statement of the theorem is established by the homeomorphism discussed above  between $cl(X)$ equipped with the topology induced by the Hausdorff distance and the Fell topology on $Ideals(C(X)).$
\end{proof}

\bibliographystyle{amsplain}
\bibliography{thesis}

\providecommand{\bysame}{\leavevmode\hbox to3em{\hrulefill}\thinspace}
\providecommand{\MR}{\relax\ifhmode\unskip\space\fi MR }
\providecommand{\MRhref}[2]{%
  \href{http://www.ams.org/mathscinet-getitem?mr=#1}{#2}
}
\providecommand{\href}[2]{#2}
\begin{thebibliography}{10}

\bibitem{Aguilar16c}
Konrad Aguilar, \emph{Fell topologies for {AF}-algebras and the quantum
  propinquity}, J. Operator Theory \textbf{82} (2019), no.~2, 469--514, arXiv:
  1608.07016. \MR{4015960}

\bibitem{Aguilar18}
\bysame, \emph{Inductive limits of {$\rm C^*$}-algebras and compact quantum
  metric spaces}, J. Aust. Math. Soc. \textbf{111} (2021), no.~3, 289--312,
  arXiv:1807.10424. \MR{4337940}

\bibitem{Aguilar}
Konrad Aguilar and Fr\'{e}d\'{e}ric Latr\'{e}moli\`ere, \emph{Quantum
  ultrametrics on {AF} algebras and the {G}romov-{H}ausdorff propinquity},
  Studia Math. \textbf{231} (2015), no.~2, 149--193. \MR{3465284}

\bibitem{Aguilar21}
Konrad Aguilar, Fr\'{e}d\'{e}ric Latr\'{e}moli\`ere, and Timothy Rainone,
  \emph{Bunce-{D}eddens algebras as quantum {G}romov-{H}ausdorff distance
  limits of circle algebras}, Integral Equations Operator Theory \textbf{94}
  (2022), no.~1, Paper No. 2, 42. \MR{4353489}

\bibitem{Beer}
Gerald Beer, \emph{Topologies on closed and closed convex sets}, Mathematics
  and its Applications, vol. 268, Kluwer Academic Publishers Group, Dordrecht,
  1993. \MR{1269778}

\bibitem{Fell61}
J.~M.~G. Fell, \emph{The structure of algebras of operator fields}, Acta Math.
  \textbf{106} (1961), 233--280. \MR{164248}

\bibitem{Fell62}
\bysame, \emph{A {H}ausdorff topology for the closed subsets of a locally
  compact non-{H}ausdorff space}, Proc. Amer. Math. Soc. \textbf{13} (1962),
  472--476. \MR{0139135}

\bibitem{Kantorovich58}
{L}.~{V}. {K}antorovich and {G}.~{Sh}. {R}ubinstein, \emph{On the space of
  completely additive functions}, Vestnik Leningrad Univ., Ser. Mat. Mekh. i
  Astron. \textbf{13} (1958), no.~7, 52--59, In Russian.

\bibitem{Kantorovich40}
L.~Kantorovitch, \emph{A new method of solving of some classes of extremal
  problems}, C. R. (Doklady) Acad. Sci. URSS (N.S.) \textbf{28} (1940),
  211--214. \MR{0003456}

\bibitem{Kerr02}
D.~{K}err, \emph{Matricial quantum {G}romov-{H}ausdorff distance}, J. Funct.
  Anal. \textbf{205} (2003), no.~1, 132--167, math.OA/0207282.

\bibitem{Kerr09}
David Kerr and Hanfeng Li, \emph{On {G}romov-{H}ausdorff convergence for
  operator metric spaces}, J. Operator Theory \textbf{62} (2009), no.~1,
  83--109. \MR{2520541}

\bibitem{Latremoliere13c}
Fr\'{e}d\'{e}ric Latr\'{e}moli\`ere, \emph{Convergence of fuzzy tori and
  quantum tori for the quantum {G}romov-{H}ausdorff propinquity: an explicit
  approach}, M\"{u}nster J. Math. \textbf{8} (2015), no.~1, 57--98.
  \MR{3549521}

\bibitem{Latremoliere13}
\bysame, \emph{The quantum {G}romov-{H}ausdorff propinquity}, Trans. Amer.
  Math. Soc. \textbf{368} (2016), no.~1, 365--411. \MR{3413867}

\bibitem{Modular}
\bysame, \emph{The modular {G}romov-{H}ausdorff propinquity}, Dissertationes
  Math. \textbf{544} (2019), 70. \MR{4036723}

\bibitem{Latremoliere2020b}
\bysame, \emph{Convergence of {H}eisenberg modules over quantum 2-tori for the
  modular {G}romov-{H}ausdorff propinquity}, J. Operator Theory \textbf{84}
  (2020), no.~1, 211--237. \MR{4157360}

\bibitem{Latremoliere18}
\bysame, \emph{The covariant {G}romov-{H}ausdorff propinquity}, Studia Math.
  \textbf{251} (2020), no.~2, 135--169. \MR{4045657}

\bibitem{Latremoliere21b}
\bysame, \emph{The dual modular {G}romov-{H}ausdorff propinquity and
  completeness}, J. Noncommut. Geom. \textbf{15} (2021), no.~1, 347--398.
  \MR{4248216}

\bibitem{Latremoliere21}
\bysame, \emph{The {G}romov-{H}ausdorff propinquity for metric spectral
  triples}, Adv. Math. \textbf{404} (2022), Paper No. 108393, 56. \MR{4411527}

\bibitem{Li03}
H.~{L}i, \emph{{$C^\ast$}-algebraic quantum {G}romov-{H}ausdorff distance},
  (2003), ArXiv: math.OA/0312003.

\bibitem{mcshane}
E.~J. McShane, \emph{Extension of range of functions}, Bull. Amer. Math. Soc.
  \textbf{40} (1934), no.~12, 837--842. \MR{1562984}

\bibitem{Pedersen79}
{G}.~{K}. {P}edersen, \emph{{C}*-{A}lgebras and their automorphism groups},
  Academic Press, 1979.

\bibitem{Rieffel74}
M.~A. {R}ieffel, \emph{Induced representations of {C*}-algebras}, Advances in
  Math. \textbf{13} (1974), 176--257.

\bibitem{Rieffel98a}
\bysame, \emph{Metrics on states from actions of compact groups}, Doc. Math.
  \textbf{3} (1998), 215--229. \MR{1647515}

\bibitem{Rieffel00}
\bysame, \emph{Gromov-{H}ausdorff distance for quantum metric spaces}, vol.
  168, 2004, Appendix 1 by Hanfeng Li, Gromov-Hausdorff distance for quantum
  metric spaces. Matrix algebras converge to the sphere for quantum
  Gromov-Hausdorff distance, pp.~1--65. \MR{2055927}

\bibitem{Rieffel01}
\bysame, \emph{Matrix algebras converge to the sphere for quantum
  {G}romov--{H}ausdorff distance}, Mem. Amer. Math. Soc. \textbf{168} (2004),
  no.~796, 67--91, math.OA/0108005.

\bibitem{Rieffel15}
\bysame, \emph{Matricial bridges for ``matrix algebras converge to the
  sphere''}, Operator algebras and their applications, Contemp. Math., vol.
  671, Amer. Math. Soc., Providence, RI, 2016, ArXiv: 1502.00329, pp.~209--233.

\bibitem{Rieffel}
Marc~A. Rieffel, \emph{Metrics on states from actions of compact groups}, Doc.
  Math. \textbf{3} (1998), 215--229. \MR{1647515}

\bibitem{Rieffel18}
\bysame, \emph{Vector bundles for ``matrix algebras converge to the sphere''},
  J. Geom. Phys. \textbf{132} (2018), 181--204. \MR{3836776}

\bibitem{Wu}
Wei Wu, \emph{Quantized {G}romov-{H}ausdorff distance}, J. Funct. Anal.
  \textbf{238} (2006), no.~1, 58--98. \MR{2234123}

\end{thebibliography}
\vfill

\end{document}